\documentclass[reqno]{amsart}
\usepackage[OT1]{fontenc}
\usepackage{amsmath,amssymb,mathtools,mathrsfs,booktabs}
\usepackage[hidelinks]{hyperref}
\numberwithin{equation}{section}
\newtheorem{theorem}{Theorem}[section]
\newtheorem{proposition}[theorem]{Proposition}
\newtheorem{lemma}[theorem]{Lemma}
\theoremstyle{definition}
\newtheorem{example}[theorem]{Example}
\theoremstyle{remark}
\newtheorem{remark}[theorem]{Remark}
\DeclareMathOperator{\Dom}{Dom}
\DeclareMathOperator{\Tr}{Tr}

\newcommand{\N}{\mathbb N}
\newcommand{\R}{\mathbb R}
\newcommand{\C}{\mathbb C}
\newcommand{\D}{\mathcal D}
\newcommand{\ds}{\mathcal N}
\newcommand{\TC}{\mathcal S_1}
\newcommand{\ip}[2]{\langle #1,#2\rangle}
\newcommand{\norm}[1]{\lVert #1\rVert}
\title{Unbounded Helson Forms with a Moment Representation}
\author{Tengfei Ma}
\address{School of Mathematical Sciences, Dalian University of Technology,
Dalian, Liaoning 116024, P. R. China}
\email{503304283@qq.com}
\author{Yufeng Lu}
\address{School of Mathematical Sciences, Dalian University of Technology,
Dalian, Liaoning 116024, P. R. China}
\email{lyfdlut@dlut.edu.cn}
\author[Chao Zu]{Chao Zu$^*$}
\address{School of Mathematical Sciences, Dalian University of Technology,
Dalian, Liaoning 116024, P. R. China}
\email{zuchao@dlut.edu.cn}
\thanks{$^*$Corresponding author.}
\date{}
\begin{document}
\begin{abstract}
We study the positive Helson form
$h_0[a,b]=\sum_{m,n\ge1}\alpha(mn)a_n\overline{b_m}$
on the finitely supported sequences $c_{00}\subset\ell^2(\N)$,
where $\alpha(n)=\int_{\R}n^{-s}\,d\mu(s)<\infty$ for a finite
positive Borel measure $\mu$. Under the Carleman condition
$\sum_{k\ge1}\alpha(4^k)^{-1/(2k)}=\infty$, closability of $h_0$
is equivalent to $\sqrt n\,\alpha(n)\to0$, and hence to
$\mu((-\infty,1/2])=0$. For measures concentrated on
$(1/2,\infty)$, we identify the closure of the operator
$N_0a=\sum_{n\ge1}a_nn^{-s}$, defined on $c_{00}$,
with its maximal realization from $\ell^2(\N)$ to $L^2(\mu)$.
We further prove that the sequences $a\in c_{00}$ satisfying
$\sum_{n\ge1}a_n/\sqrt n=0$ form an operator core for the
positive self-adjoint operator associated with the closure of $h_0$.
\end{abstract}
\maketitle
\enlargethispage{4pt} % Allow for the corresponding-author footnote.

\section{Introduction}\label{sec:introduction}

Let $M$ be a finite Borel measure on $\R$ with finite moments of all orders,
and put
\[
 m_j=\int_{\R}x^j\,dM(x),\qquad j\ge0.
\]
The corresponding Hankel form is
\[
 Q[a,b]=\sum_{j,k\ge0}m_{j+k}a_k\overline{b_j},
 \qquad a,b\in c_{00}(\mathbb N_0),
\]
where $\mathbb N_0=\{0,1,\ldots\}$ and $c_{00}(\mathbb N_0)$ denotes the
finitely supported sequences.
Widom \cite[Theorems~3.1 and~3.2]{W66} characterized the boundedness and
compactness of positive Hankel forms in terms of their moments:
boundedness is equivalent to $m_n=O(n^{-1})$, and compactness to
$m_n=o(n^{-1})$. 

The study of unbounded positive Hankel operators through
moment representations and closable quadratic forms was initiated by
Yafaev \cite{Y16}. He investigated the relation between closability and
the representing measure and described the domains of the closed forms,
with the closability criterion subsequently corrected under additional
growth assumptions in \cite{Y19}. Berg and Szwarc established closability
for further classes of moment sequences in \cite{BS20} and studied the
domains and operator cores of the associated positive self-adjoint Hankel
operators in \cite{BS22}.
Related questions concerning operator domains and cores have also
been studied for Hankel operators with $H^2$ symbols
\cite[Theorem~1.1]{GP23} and for integral Hankel operators
\cite{PT25}.

The aim of the present paper is to develop the corresponding theory for Helson
forms, also called multiplicative Hankel forms. These forms were studied by
Helson \cite{H06,H10} in connection with Hankel forms on the
infinite-dimensional torus. An important example is the Helson form
\[
 \sum_{m,n\ge2}
 \frac{a_n\overline{b_m}}{\sqrt{mn}\log(mn)},
 \qquad a,b\in c_{00}(\{2,3,\ldots\}),
\]
associated with the multiplicative Hilbert matrix.
Its norm is $\pi$, and its spectrum is $[0,\pi]$ with no eigenvalues
\cite[Theorem~1]{BPSSV16}. The spectrum was subsequently shown to be
purely absolutely continuous of multiplicity one
\cite[Theorem~1.1]{PP18b}.

Let $\mu$ be a finite positive Borel measure on
$\R$ such that
\begin{equation}\label{eq:moments}
 \alpha(n)=\int_{\R}n^{-s}\,d\mu(s)<\infty,\qquad n\in\N=\{1,2,\ldots\}.
\end{equation}
The associated Helson form is
\begin{equation}\label{eq:initial-form}
 h_0[a,b]=\sum_{m,n\ge1}\alpha(mn)a_n\overline{b_m},
 \qquad a,b\in c_{00},
\end{equation}
where $c_{00}$ denotes the finitely supported sequences in $\ell^2(\N)$.
Such moment representations were studied by Perfekt and Pushnitski
\cite[Section~5.1]{PP18a}, who characterized the boundedness of the corresponding
positive Helson forms and considered moment representations in infinitely many
variables.

To study the closure of $h_0$, we introduce the operator
\[
 N_0:\ell^2(\N)\supset c_{00}\longrightarrow L^2(\mu),\qquad
 (N_0a)(s)=\sum_{n\ge1}a_nn^{-s}.
\]
The moment assumption ensures that $N_0$ is well defined, and
\[
 h_0[a,b]=\ip{N_0a}{N_0b}_{L^2(\mu)},\qquad a,b\in c_{00}.
\]
In particular, $h_0$ is positive.

For $a\in\ell^2(\N)$, consider the Dirichlet series
\begin{equation}\label{eq:expression}
 (\ds a)(s)=\sum_{n\ge1}a_nn^{-s},
\end{equation}
which belongs to the Hilbert space of Dirichlet series studied
in \cite{HLS97}. It converges absolutely for $s>1/2$, since
\begin{equation}\label{eq:evaluation}
 \sum_{n\ge1}|a_n|n^{-s}
 \le\norm a_{\ell^2}\,\zeta(2s)^{1/2}.
\end{equation}
The point $s=1/2$ is therefore the natural boundary for bounded evaluation of
these series: the sequence $(n^{-s})_{n\ge1}$ belongs to $\ell^2(\N)$ precisely
when $s>1/2$.

Our first main theorem characterizes the closability of $h_0$ under a Carleman
condition in terms of the support of the representing measure $\mu$.

\begin{theorem}\label{thm:support}
Assume \eqref{eq:moments} is satisfied. 
Then the following conditions are equivalent:
\begin{enumerate}
\item\label{item:support-closable} the form $h_0$ defined on $c_{00}$ is closable in $\ell^2(\N)$ and 
\begin{equation}\label{eq:carleman}
 \sum_{k=1}^{\infty}\alpha(4^k)^{-1/(2k)}=\infty;
\end{equation}
\item\label{item:support-decay} $\sqrt n\,\alpha(n)\to0$ as $n\to\infty$;
\item\label{item:support-measure} $\mu(( -\infty,1/2])=0$.
\end{enumerate}
\end{theorem}

The Carleman's condition is needed only for the implication from closability to
the support restriction. In fact, closability alone does not force the measure to
be concentrated on $(1/2,\infty)$; an example is given in Section~\ref{subsec:support}.
Theorem~\ref{thm:support} remains valid if the Carleman condition
\eqref{eq:carleman} is replaced by
\[
 \sum_{k=1}^{\infty}\alpha(b^{2k})^{-1/(2k)}=\infty,
\]
where $b\ge2$ is any fixed integer. Each of these conditions implies
\eqref{eq:carleman}, so $b=2$ gives the weakest assumption in this family;
see Remark~\ref{rem:bases}.

The proof of  (\ref{item:support-closable})$\Rightarrow$(\ref{item:support-measure}) begins by restricting $h_0$ to sequences
supported on $\{2^k:k\ge0\}$. The resulting Hankel moment form satisfies
Carleman's condition, so the classical closability theorem gives
$\mu((-\infty,0])=0$. To obtain the sharper restriction at $1/2$, we use the
square summability of the Laplace transform of $f\,d\mu$ at the points $\log n$,
for $f\in\Dom(N_0^*)$. The Euler summation formula gives a weighted $L^2$ estimate for
this transform, which permits analytic continuation of the corresponding
Cauchy transform and shows that $f=0$ $\mu$-almost everywhere on $(0,1/2]$.
Density of $\Dom(N_0^*)$ in $L^2(\mu)$ then yields $\mu((-\infty,1/2])=0$.

Assume now that $\mu((-\infty,1/2])=0$. Since \eqref{eq:expression} is then defined
$\mu$-almost everywhere for every $a\in\ell^2$, we introduce the maximal realization
\begin{equation*}
 \Dom(N_{\max})=\{a\in\ell^2:\ds a\in L^2(\mu)\},\qquad N_{\max}a=\ds a.
\end{equation*}

Our second main theorem identifies $\overline{N_0}$ with $N_{\max}$ and describes
the closure of the Helson form $h_0$.

\begin{theorem}\label{thm:synthesis}
Let $\mu$ be a finite positive measure satisfying
\begin{equation}\label{eq:concentration}
 \mu((-\infty,1/2])=0.
\end{equation}
Then
\[
 \overline{N_0}=N_{\max}.
\]
Consequently the closure $h$ of $h_0$ is given by
\[
h[a,b]=\ip{N_{\max}a}{N_{\max}b}_{L^2(\mu)},\qquad  \Dom(h)=\Dom(N_{\max}) .
\]
\end{theorem}

The main step in the proof of Theorem~\ref{thm:synthesis} is to identify the
closure of $N_0$ with its maximal realization. Since
$\Dom(h)=\Dom(\overline{N_0})$ and $\overline{N_0}=N_0^{**}$, it suffices to prove
\[
 N_0^{**}=N_{\max}.
\]
The essential step is the inclusion $N_{\max}\subset N_0^{**}$, which amounts to
verifying
\[
 \ip{N_{\max}a}{f}_{L^2(\mu)}=\ip{a}{N_0^*f}_{\ell^2}
\]
for every $a\in\Dom(N_{\max})$ and $f\in\Dom(N_0^*)$. Written explicitly,
this identity becomes
\[
 \int_{1/2}^{\infty}
 \left(\sum_{n=1}^{\infty}a_nn^{-s}\right)\overline{f(s)}\,d\mu(s)
 =\sum_{n=1}^{\infty}a_n
 \left(\int_{1/2}^{\infty}n^{-s}\overline{f(s)}\,d\mu(s)\right).
\]
Absolute convergence of the Dirichlet series at each $s>1/2$ does not by itself
justify interchanging the sum and the integral. For Hankel moment forms, Yafaev's proof in \cite{Y16} relies on
a maximal realization theorem for the Laplace transform established
in \cite[Theorem~3.9]{Y15}. Berg and Szwarc subsequently gave a
moment-theoretic proof \cite[proof of Theorem~1.2, pp.~5--7]{BS22}. In the present setting, we use a different argument based on an exponential
cutoff approximation. More precisely, we introduce
\[
q_N(s)=1-e^{-N(s-1/2)}
\]
and first establish the required identity for the truncated functions
\(q_Nf\). The desired formula is then obtained by passing to the limit in
the graph norm of \(N_0^*\), where the convergence follows from the
translation property of the associated Laplace transform.
For completeness, Appendix~\ref{app:hankel-maximal} presents an alternative
proof of the corresponding maximal realization result for Hankel moment
operators. The proof is based on the same cutoff approximation principle.

Let $H_\mu$ denote the positive self-adjoint operator associated with the closed
form $h$, and put
\begin{equation*}
 \D_0=\left\{a\in c_{00}:\sum_{n\ge1}\frac{a_n}{\sqrt n}=0\right\}.
\end{equation*}

Our third main theorem establishes $\D_0$ as an operator core for $H_\mu$.

\begin{theorem}\label{thm:core-main}
Under the assumptions of Theorem~\ref{thm:synthesis}, $\D_0$ is an operator core
for $H_\mu$.
\end{theorem}

Several consequences are obtained from the same realization. In particular,
\[
 H_\mu\in\TC\iff\int_{1/2}^{\infty}\zeta(2s)\,d\mu(s)<\infty,
\]
and if $(\alpha(n))_{n\geq 1}\in\ell^2$, then $H_\mu$ agrees with the maximal operator defined by
the matrix $(\alpha(mn))_{m,n\ge1}$.

The paper is organized as follows. Section~\ref{sec:forms} recalls the required facts about
positive closed forms, proves Theorem~\ref{thm:support}, and identifies the maximal
Dirichlet-series realization. Section~\ref{sec:operators} studies the associated self-adjoint operator,
its operator core and bounded approximation, as well as trace class and the maximal
matrix realization.

\section{Closable Helson forms and their closure}\label{sec:forms}

\subsection{Preliminaries}
All inner products are linear in the first variable. Unless another space is
indicated, $\norm{\cdot}$ denotes the norm in $\ell^2=\ell^2(\N)$, and $(e_n)_{n\ge1}$
is its standard basis.

We begin with some standard facts about positive forms; see
\cite[Chapter~10]{S12} for a general treatment and
\cite[Section~2]{BS22} for their use in the setting of Hankel
moment forms.

Let $\mathcal H$ be a complex Hilbert space. A densely defined positive hermitian
form $q$ has a dense domain $\Dom(q)\subset\mathcal H$ and satisfies
\[
 q[u,v]=\overline{q[v,u]},\qquad q[u,u]\ge0.
\]
It is closed if $\Dom(q)$ is complete for the form norm
\begin{equation}\label{eq:form-norm}
 \norm u_q=\bigl(\norm u_{\mathcal H}^2+q[u,u]\bigr)^{1/2}.
\end{equation}
A positive form is closable if it has a closed extension. Equivalently, for every sequence $(u_j)_{j\ge1}\subset\Dom(q)$ satisfying
\begin{equation}\label{eq:closable-criterion}
 u_j\longrightarrow0 \quad\text{in }\mathcal H\quad(j\to\infty),
 \qquad
 q[u_j-u_k,u_j-u_k]\longrightarrow0\quad(j,k\to\infty),
\end{equation}
we have $q[u_j,u_j]\to0$ as $j\to\infty$.
In this case the smallest closed extension is denoted by $\overline q$.

We shall use the representation theorem in the following form. If $q$ is a densely
defined positive closed form, there is a unique positive self-adjoint operator $H$
such that
\begin{equation*}
 \Dom(q)=\Dom(H^{1/2}),\qquad q[u,v]=\ip{H^{1/2}u}{H^{1/2}v}.
\end{equation*}
and,
\begin{equation}\label{eq:representation-action}
 q[u,v]=\ip{Hu}{v},\qquad u\in\Dom(H),\quad v\in\Dom(q).
\end{equation}
Thus $u\in\Dom(H)$ exactly when $u\in\Dom(q)$ and the functional
$v\mapsto q[u,v]$ is bounded in the norm of $\mathcal H$.

A subspace of $\Dom(q)$ is a form core if it is dense for the norm
\eqref{eq:form-norm}. If $A$ is a closed operator, a subspace of $\Dom(A)$ is an
operator core if it is dense for the graph norm $(\norm u^2+\norm{Au}^2)^{1/2}$.
These two notions will be used separately below.

The form \eqref{eq:initial-form} is generated by the operator $N_0$. The following elementary observation will be used repeatedly; its proof is straightforward and is omitted.

\begin{lemma}\label{lem:form-operator}
Let $T:\mathcal H\supset\Dom(T)\to\mathcal K$ be densely defined and set
\[
 q_T[u,v]=\ip{Tu}{Tv}_{\mathcal K}.
\]
Then $q_T$ is closable if and only if $T$ is closable. If this holds, then
\[
 \Dom(\overline{q_T})=\Dom(\overline T),\qquad
 \overline{q_T}[u,v]=\ip{\overline T u}{\overline T v}_{\mathcal K}.
\]
\end{lemma}

By Lemma~\ref{lem:form-operator}, closability of $h_0$ is equivalent to
closability of $N_0$, and hence to density of $\Dom(N_0^*)$ in $L^2(\mu)$.
The following lemma gives the domain and action of $N_0^*$.

\begin{lemma}\label{lem:adjoint}
The adjoint $N_0^*:L^2(\mu)\to\ell^2$ has domain
\begin{equation*}
 \Dom(N_0^*)=\left\{f\in L^2(\mu):
 \left(\int_{\R}f(s)n^{-s}\,d\mu(s)\right)_{n\ge1}\in\ell^2\right\},
\end{equation*}
and
\begin{equation}\label{eq:adjoint-domain}
 (N_0^*f)_n=\int_{\R}f(s)n^{-s}\,d\mu(s).
\end{equation}
\end{lemma}
\begin{proof}
For fixed $n$,
\[
 \int_{\R}|f(s)|n^{-s}\,d\mu(s)\le\norm f_{L^2(\mu)}\,\alpha(n^2)^{1/2},
\]
so the scalar integrals in \eqref{eq:adjoint-domain} are absolutely convergent.
For $a\in c_{00}$,
\[
 \ip{N_0a}{f}_{L^2(\mu)}=\sum_n a_n\overline{\int_{\R}f(s)n^{-s}\,d\mu(s)}.
\]
This functional is bounded in $a\in c_{00}$ precisely when the sequence in
\eqref{eq:adjoint-domain} belongs to $\ell^2$, proving the assertions.
\end{proof}

\subsection{Closability and the support of the measure}\label{subsec:support}
The next lemma follows from a result on closable Hankel moment forms.
\begin{lemma}[{\cite[Theorem~1.1]{Y19}}]\label{lem:hankel-support}
Let $M$ be a positive Borel measure on $\R$ with finite moments of all
orders, and put
\[
 q_j=\int_{\R}x^j\,dM(x),\qquad j\ge0.
\]
Consider the operator
\[
 A_0:\ell^2(\N_0)\supset c_{00}(\N_0)\longrightarrow L^2(M),
 \qquad (A_0u)(x)=\sum_{k\ge0}u_kx^k.
\]
If $A_0$ is closable and
\[
 \sum_{k\ge1}q_{2k}^{-1/(2k)}=\infty,
\]
then
\[
 M(\R\setminus(-1,1))=0.
\]
\end{lemma}

\begin{lemma}\label{lem:dyadic}
Suppose $h_0$ is closable and, for some integer $b\ge2$,
\begin{equation}\label{eq:carleman-b}
 \sum_{k\ge1}\alpha(b^{2k})^{-1/(2k)}=\infty.
\end{equation}
Then
\[
 \mu((-\infty,0])=0.
\]
\end{lemma}
\begin{proof}
Let $M$ be the positive Borel measure on $\R$ defined by
\[
 M(E)=\mu\bigl(\{s\in\R:b^{-s}\in E\}\bigr)
\]
for every Borel set $E\subset\R$. Its moments are
\[
 q_k=\int_{\R}x^k\,dM(x)
 =\int_{\R}b^{-ks}\,d\mu(s)
 =\alpha(b^k),\qquad k\ge0.
\]
The map
\[
 U:L^2(M)\to L^2(\mu),\qquad Uf(s)=f(b^{-s}),
\]
is unitary. Let
\[
 J:\ell^2(\{0,1,\ldots\})\to\ell^2(\N),\qquad Je_k=e_{b^k}.
\]
Then $UA_0=N_0J$ on finitely supported sequences. If $u_j\to0$ and $A_0u_j\to v$,
closability of $N_0$ gives $Uv=0$, so $A_0$ is closable.
By Lemma~\ref{lem:hankel-support}, $M([1,\infty))=0$, which is equivalent to the asserted
conclusion.
\end{proof}

To prove the stronger conclusion
$\mu((-\infty,1/2])=0$, we use the following estimate relating the values of a
Laplace transform at $\log n$, $n\ge1$, to its weighted $L^2$-norm.

\begin{lemma}\label{lem:sampling}
Let $\mu$ be a finite positive measure concentrated on $(0,\infty)$.
For $f\in L^2(\mu)$, define
\[
 F(t)=\int_0^\infty f(s)e^{-st}\,d\mu(s),\qquad t\ge0.
\]
Then
\begin{equation}\label{eq:sampling}
 \norm f_{L^2(\mu)}^2+\sum_{n\ge1}|F(\log n)|^2
 \asymp_\mu
 \norm f_{L^2(\mu)}^2+\int_0^\infty e^t|F(t)|^2\,dt.
\end{equation}
Both quantities are allowed to be infinite.
\end{lemma}

\begin{proof}
Since $\mu$ is finite and $f\in L^2(\mu)$, the Cauchy--Schwarz
inequality gives
\[
 |F(t)|\le \norm{f}_{L^1(\mu)}
 \le \mu(\R)^{1/2}\norm{f}_{L^2(\mu)},\qquad t\ge0.
\]
The dominated convergence theorem shows that $F$ is continuous
on $[0,\infty)$. For $\varepsilon>0$, the estimate
\[
 s|f(s)|e^{-st}\le \frac{|f(s)|}{e\varepsilon},
 \qquad s>0,\quad t\ge\varepsilon,
\]
justifies differentiation under the integral sign and gives
\[
 F'(t)=-\int_0^\infty sf(s)e^{-st}\,d\mu(s),
 \qquad t>0.
\]
By the triangle inequality and Tonelli's theorem,
\[
 \begin{aligned}
 \int_0^\infty |F'(t)|\,dt
 &\le \int_0^\infty\int_0^\infty
       s|f(s)|e^{-st}\,d\mu(s)\,dt\\
 &=\int_0^\infty |f(s)|
       \left(\int_0^\infty se^{-st}\,dt\right)d\mu(s)
 =\norm{f}_{L^1(\mu)}.
 \end{aligned}
\]
For $0<\varepsilon<t$, the fundamental theorem of calculus gives
$F(t)-F(\varepsilon)=\int_\varepsilon^t F'(u)\,du$.
Since $F$ is continuous at zero and $F'\in L^1(0,\infty)$,
letting $\varepsilon\downarrow0$ yields
\[
 F(t)-F(0)=\int_0^t F'(u)\,du.
\]
Hence $F$ is absolutely continuous on every bounded interval
in $[0,\infty)$.

Put $\varphi(x)=|F(\log x)|^2$, $x\ge1$. This function is locally absolutely
continuous, and
\[
 \varphi'(x)=\frac2x\operatorname{Re}
 \bigl(F'(\log x)\overline{F(\log x)}\bigr)
\]
for almost every $x>1$. Since $|F(t)|\le\norm f_{L^1(\mu)}$,
\[
 \begin{aligned}
 \int_1^\infty|\varphi'(x)|\,dx
 &\le2\int_0^\infty|F(t)|\,|F'(t)|\,dt\\
 &\le2\norm f_{L^1(\mu)}^2
 \le2\mu(\R)\norm f_{L^2(\mu)}^2.
 \end{aligned}
\]
For each integer $K\ge1$, the first-order Euler summation formula gives
\[
 \sum_{n=1}^K\varphi(n)-\int_1^{K+1}\varphi(x)\,dx
 =-\int_1^{K+1}(1-\{x\})\varphi'(x)\,dx,
\]
where $\{x\}$ denotes the fractional part of $x$,
see \cite[Theorem~7.13, pp.~149--150]{A74}. Consequently, after the
substitution $x=e^t$,
\begin{equation}\label{eq:euler-estimate}
 \left|\sum_{n=1}^K|F(\log n)|^2
 -\int_0^{\log(K+1)}e^t|F(t)|^2\,dt\right|
 \le2\mu(\R)\norm f_{L^2(\mu)}^2.
\end{equation}
The asserted comparison follows from \eqref{eq:euler-estimate} by the monotone
convergence theorem.
\end{proof}

\begin{proof}[Proof of Theorem~\ref{thm:support}]
 (\ref{item:support-closable})$\Rightarrow$(\ref{item:support-measure}):   Assume first that $h_0$ is closable and that \eqref{eq:carleman} holds. By
Lemma~\ref{lem:dyadic}, with $b=2$, the measure $\mu$ is concentrated on
$(0,\infty)$. Let $f\in\Dom(N_0^*)$, and let $F$ be as in
Lemma~\ref{lem:sampling}. Then \eqref{eq:adjoint-domain} and \eqref{eq:sampling}  give
\[
 (F(\log n))_{n\ge1}\in\ell^2, \qquad e^{\frac t 2} F(t) \in L^2(0,\infty). 
\]
Consequently
\[
 G(z)=\int_0^\infty e^{-zt}F(t)\,dt
\]
is holomorphic in
\[
 \Omega=\{z\in\C:\operatorname{Re}z>-1/2\}.
\]
Let $\eta$ be the finite complex measure defined by $d\eta=f\,d\mu$.
For $\operatorname{Re}z>0$, Tonelli's theorem gives
\[
 \begin{aligned}
 \int_0^\infty\int_0^\infty
 |f(s)|e^{-(\operatorname{Re}z+s)t}\,d\mu(s)\,dt
 &=\int_0^\infty\frac{|f(s)|}{\operatorname{Re}z+s}\,d\mu(s)\\
 &\le\frac{\norm f_{L^1(\mu)}}{\operatorname{Re}z}<\infty.
 \end{aligned}
\]
Fubini's theorem therefore yields
\begin{equation*}
 \begin{aligned}
 G(z)
 &=\int_0^\infty e^{-zt}
 \left(\int_0^\infty f(s)e^{-st}\,d\mu(s)\right)dt\\
 &=\int_0^\infty f(s)
 \left(\int_0^\infty e^{-(z+s)t}\,dt\right)d\mu(s)\\
 &=\int_0^\infty\frac{f(s)}{z+s}\,d\mu(s)
 =\int_0^\infty\frac{d\eta(s)}{z+s}.
 \end{aligned}
\end{equation*}
Let $\sigma$ be the image of $\eta$ under the map $s\mapsto-s$, its Cauchy transform
\[
 C_\sigma(z)=\int_{-\infty}^0\frac{d\sigma(y)}{z-y}
\]
is holomorphic in $\C\setminus(-\infty,0]$. Since $C_\sigma(z)=G(z)$ for
$\operatorname{Re}z>0$, the identity theorem gives
\[
 C_\sigma(z)=G(z),\qquad z\in\Omega\setminus(-1/2,0].
\]
As $G$ is holomorphic in $\Omega$, the Cauchy-transform result
recalled in \cite[Section~1, p.~21]{L20}, applied locally, yields
\[
 \sigma|_{(-1/2,0]}=0.
\]
Consequently, $\eta|_{[0,1/2)}=0$, and hence $f=0$ $\mu$-almost everywhere
on $[0,1/2)$. Then we have
\[
 e^{t/2}F(t)=\int_{[1/2,\infty)}e^{-t(s-1/2)}f(s)\,d\mu(s).
\]
Since $f\in L^1(\mu)$ and $0\le e^{-t(s-1/2)}\le1$ for $s\ge1/2$,
the dominated convergence theorem gives
\[
 \lim_{t\to\infty}e^{t/2}F(t)=\int_{\{1/2\}}f(s)\,d\mu(s).
\]
The function $e^{t/2}F(t)$ belongs to $L^2(0,\infty)$, so its limit at
infinity must be zero. Hence
\[
 \int_{\{1/2\}}f(s)\,d\mu(s)=0.
\]
Combining this with the vanishing of $f$ on $[0,1/2)$, we obtain
\[
 \int_{[0,1/2]}f(s)\,d\mu(s)=0,\qquad f\in\Dom(N_0^*).
\]
Thus $\mathbf 1_{[0,1/2]}$ is orthogonal to the dense subspace $\Dom(N_0^*)$
of $L^2(\mu)$, and consequently $\mu([0,1/2])=0$.
Together with Lemma~\ref{lem:dyadic}, this yields $\mu((-\infty,1/2])=0$,
proving (\ref{item:support-closable})$\Rightarrow$(\ref{item:support-measure}).

(\ref{item:support-decay}) $\Leftrightarrow$ (\ref{item:support-measure}) is obvious.

(\ref{item:support-decay}), (\ref{item:support-measure}) $\Rightarrow$ (\ref{item:support-closable}):   Assume now that $\mu$ is concentrated on $(1/2,\infty)$. Let $a_j\in c_{00}$
satisfy $a_j\to0$ in $\ell^2$ and $N_0a_j\to f$ in $L^2(\mu)$.
By \eqref{eq:evaluation},
\[
 |(N_0a_j)(s)|\le\zeta(2s)^{1/2}\|a_j\|_{\ell^2}\longrightarrow0,
 \qquad s>1/2.
\]
The convergence in $L^2(\mu)$ also yields a subsequence converging to $f$
$\mu$-almost everywhere. Since $\mu((-\infty,1/2])=0$, the two limits imply
that $f=0$ $\mu$-almost everywhere. Thus $N_0$ is closable, and
Lemma~\ref{lem:form-operator} gives the closability of $h_0$. Finally, the implication (\ref{item:support-decay}) $\Rightarrow$ \eqref{eq:carleman} is  immediate.

\end{proof}

\begin{remark}\label{rem:bases}
Theorem~\ref{thm:support} remains valid if \eqref{eq:carleman} is replaced by
\eqref{eq:carleman-b} for any fixed integer $b\ge2$. More generally, the
condition with base $q$ implies the condition with base $p$ whenever
$p,q$ are integers with $2\le p<q$. Thus $b=2$ gives the weakest condition
in this family.

The case $\mu=0$ is immediate. Suppose that $\mu\ne0$, and put
\[
 A(t)=\int_{\R}e^{-ts}\,d\mu(s),\qquad t\ge0.
\]
For each $t\ge0$, choose an integer $N\ge1$ such that $\log N\ge t$.
Since $e^{-ts}\le N^{-s}$ for $s\le0$ and $e^{-ts}\le1$ for $s>0$,
\[
 A(t)\le\int_{(-\infty,0]}N^{-s}\,d\mu(s)+\mu((0,\infty))
 \le\alpha(N)+\mu(\R)<\infty.
\]
For $t_0,t_1\ge0$ and $0<\theta<1$, H\"older's inequality gives
\[
 \begin{aligned}
 A((1-\theta)t_0+\theta t_1)
 &=\int_{\R}(e^{-t_0s})^{1-\theta}(e^{-t_1s})^\theta\,d\mu(s)\\
 &\le A(t_0)^{1-\theta}A(t_1)^\theta.
 \end{aligned}
\]
Hence $A$ is logarithmically convex.

Let $p,q$ be integers with $2\le p<q$, and set $\theta=\log p/\log q$.
Since $A(0)=\mu(\R)$ and
$A(\log n)=\alpha(n)$, logarithmic convexity yields
\[
 \alpha(p^{2k})=A(2k\log p)
 \le\mu(\R)^{1-\theta}\alpha(q^{2k})^\theta.
\]
Consequently,
\[
 \alpha(p^{2k})^{-1/(2k)}
 \ge\mu(\R)^{-(1-\theta)/(2k)}
 \left(\alpha(q^{2k})^{-1/(2k)}\right)^\theta.
\]
The first factor on the right tends to $1$. Moreover, for any positive
sequence $(u_k)$ and $0<\theta<1$, divergence of $\sum_k u_k$ implies
divergence of $\sum_k u_k^\theta$: either $u_k\ge1$ infinitely often, or
$u_k^\theta\ge u_k$ eventually. The preceding inequality therefore proves
the asserted implication between the two Carleman conditions.

Finally, Lemma~\ref{lem:dyadic} applies with any base $b\ge2$, and the
remaining steps in the proof of Theorem~\ref{thm:support} are independent
of that choice.
\end{remark}

The following example shows that the Carleman condition cannot be omitted
from the implication (\ref{item:support-closable})$\Rightarrow$(\ref{item:support-measure}) in Theorem~\ref{thm:support}.

\begin{example}\label{ex:counter}
Let
\[
 \mu=\sum_{j=0}^{\infty}\frac1{(j!)^2}\delta_{-j}.
\]
This measure is finite, and
\[
 \alpha(n)=\int_{\R}n^{-s}\,d\mu(s)
 =\sum_{j=0}^{\infty}\frac{n^j}{(j!)^2}<\infty,\qquad n\ge1.
\]
The associated form $h_0$ is closable, although
\[
 \mu((-\infty,1/2])=\mu(\R)>0,\qquad
 \sqrt n\,\alpha(n)\ge\sqrt n.
\]

To prove closability, we show that $\Dom(N_0^*)$ is dense in $L^2(\mu)$.
For $k\ge0$ and $0<\varepsilon<1$, define
\[
 f_{k,\varepsilon}(-j)=
 \begin{cases}
 0,&j<k,\\[2mm]
 \displaystyle(j!)^2\frac{(-\varepsilon)^{j-k}}{(j-k)!},&j\ge k.
 \end{cases}
\]
Then
\[
 \norm{f_{k,\varepsilon}}_{L^2(\mu)}^2
 =\sum_{\ell=0}^{\infty}\varepsilon^{2\ell}
 \left(\frac{(k+\ell)!}{\ell!}\right)^2<\infty,
\]
since the factorial ratio is a polynomial of degree $k$ in $\ell$.
Moreover,
\[
 \int_{\R}f_{k,\varepsilon}(s)n^{-s}\,d\mu(s)
 =\sum_{\ell=0}^{\infty}\frac{(-\varepsilon)^\ell n^{k+\ell}}{\ell!}
 =n^ke^{-\varepsilon n}.
\]
The sequence on the right belongs to $\ell^2$, so Lemma~\ref{lem:adjoint}
gives $f_{k,\varepsilon}\in\Dom(N_0^*)$.

As $\varepsilon\downarrow0$, the dominated convergence theorem gives
\[
 \norm{f_{k,\varepsilon}-(k!)^2\mathbf1_{\{-k\}}}_{L^2(\mu)}^2
 =\sum_{\ell=1}^{\infty}\varepsilon^{2\ell}
 \left(\frac{(k+\ell)!}{\ell!}\right)^2\longrightarrow0.
\]
Thus the closure of $\Dom(N_0^*)$ contains every $\mathbf1_{\{-k\}}$.
Their linear span is dense in $L^2(\mu)$, proving that $h_0$ is closable.

\end{example}

\subsection{The maximal Dirichlet-series realization}
From now on we assume only that $\mu$ is finite and satisfies
\eqref{eq:concentration}.
The operator $N_{\max}$ is closed. Indeed, let $a_j\in\Dom(N_{\max})$
satisfy $a_j\to a$ in $\ell^2$ and $N_{\max}a_j\to f$ in $L^2(\mu)$.
By \eqref{eq:evaluation},
\[
 |(\ds a_j)(s)-(\ds a)(s)|
 \le\zeta(2s)^{1/2}\norm{a_j-a}_{\ell^2}\longrightarrow0,
 \qquad s>1/2.
\]
Passing to a subsequence, we also have $\ds a_j\to f$ $\mu$-almost everywhere.
Since $\mu$ is concentrated on $(1/2,\infty)$, the two limits give
$\ds a=f$ $\mu$-almost everywhere. Thus $a\in\Dom(N_{\max})$ and
$N_{\max}a=f$, proving closedness. Moreover, the restriction of $N_{\max}$
to $c_{00}$ is $N_0$, so $\overline {N_0}\subset N_{\max}$.

To prove $\overline{N_0}=N_{\max}$, we introduce the operator
$L:L^2(\mu)\supset\Dom(L)\to L^2(0,\infty)$ defined by
\begin{equation*}
 (Lf)(t)=e^{t/2}F(t)
 =\int_{(1/2,\infty)}e^{-t(s-1/2)}f(s)\,d\mu(s),\qquad t\ge0,
\end{equation*}
where $F$ is as in Lemma~\ref{lem:sampling}, and
\[
 \Dom(L)=\{f\in L^2(\mu):Lf\in L^2(0,\infty)\}.
\]
Lemmas~\ref{lem:adjoint} and \ref{lem:sampling} then give
\begin{equation}\label{eq:graph-equivalence}
 \begin{gathered}
 \Dom(N_0^*)=\Dom(L),
 \\
 \norm f_{L^2(\mu)}^2+\norm{N_0^*f}_{\ell^2}^2
 \asymp_\mu\norm f_{L^2(\mu)}^2+\norm{Lf}_{L^2(0,\infty)}^2.
 \end{gathered}
\end{equation}

\begin{proof}[Proof of Theorem~\ref{thm:synthesis}]
Fix $a\in \Dom(N_{\max})$ and $f\in \Dom(N_0^*)$. For each integer $N\geq1$, put
\[
q_N(s)=1-e^{-N(s-1/2)},\qquad s>1/2.
\]
Since $0\leq q_N(s)\leq1$ and
$q_N(s)\leq N(s-1/2)$, we have
\begin{align*}
&\int_{1/2}^{\infty}
q_N(s)^2\zeta(2s)\,d\mu(s)\\
&\quad\leq
\int_{1/2}^{\infty}
\left(q_N(s)^2+
\frac{q_N(s)^2}{2s-1}\right)d\mu(s)\\
&\quad\leq
\left(1+\frac N2\right)\mu(\mathbb R)<\infty.
\end{align*}
By \eqref{eq:evaluation}, the Cauchy--Schwarz inequality, and
Tonelli's theorem, for every $b\in\ell^2$,
\begin{align*}
&\sum_{n=1}^{\infty}|b_n|
\int_{1/2}^{\infty}
n^{-s}q_N(s)|f(s)|\,d\mu(s)\\
&\quad=
\int_{1/2}^{\infty}
q_N(s)|f(s)|
\sum_{n=1}^{\infty}|b_n|n^{-s}\,d\mu(s)\\
&\quad\leq
\|b\|_{\ell^2}\|f\|_{L^2(\mu)}
\left(
\int_{1/2}^{\infty}
q_N(s)^2\zeta(2s)\,d\mu(s)
\right)^{1/2}<\infty.
\end{align*}
Since $q_Nf\in L^2(\mu)$, the preceding
estimate and Lemma~\ref{lem:adjoint} imply that
$q_Nf\in\operatorname{Dom}(N_0^*)$.
The absolute convergence also permits the
interchange of the sum and the integral, giving
\begin{equation}\label{eq:cutoff-interchange}
\begin{aligned}
&\langle N_{\max}a,q_Nf\rangle_{L^2(\mu)}\\
&\quad=
\sum_{n=1}^{\infty}a_n
\int_{1/2}^{\infty}
n^{-s}q_N(s)\overline{f(s)}\,d\mu(s)\\
&\quad=
\langle a,N_0^*(q_Nf)\rangle_{\ell^2}.
\end{aligned}
\end{equation}
By \eqref{eq:graph-equivalence} and the definition of $L$,
\begin{align*}
&\|f-q_Nf\|_{L^2(\mu)}^2
+\|N_0^*f-N_0^*(q_Nf)\|_{\ell^2}^2\\
&\quad\lesssim_\mu
\int_{1/2}^{\infty}
e^{-2N(s-1/2)}|f(s)|^2\,d\mu(s)
+\int_N^\infty |(Lf)(t)|^2\,dt\\
&\quad\longrightarrow0,
\end{align*}
where the first integral tends to zero by
the dominated convergence theorem and the
second by $Lf\in L^2(0,\infty)$.
Letting $N\to\infty$ in
\eqref{eq:cutoff-interchange}, we obtain
\[
\langle N_{\max}a,f\rangle_{L^2(\mu)}
=
\langle a,N_0^*f\rangle_{\ell^2},
\]
hence $a\in\operatorname{Dom}(N_0^{**})$ and
$N_0^{**}a=N_{\max}a$. Therefore, we have $N_{\max}\subset N_0^{**}$.
Since $N_0$ is closable,
$N_0^{**}=\overline{N_0}$.
Thus $N_{\max}\subset\overline{N_0}$,
and consequently
\[
\overline{N_0}=N_{\max}.
\]
Finally, Lemma~\ref{lem:form-operator} gives
\[
\operatorname{Dom}(h)=\operatorname{Dom}(N),
\qquad
h[a,b]=\langle Na,Nb\rangle_{L^2(\mu)},
\]
where $N=\overline{N_0}=N_{\max}$
and $a,b\in\operatorname{Dom}(N)$.
\end{proof}

\section{Self-adjoint operators associated with Helson forms}\label{sec:operators}

\subsection{The self-adjoint operator associated with \texorpdfstring{$h$}{h}}
Let $H_\mu$ be the positive self-adjoint operator associated with the closed
Helson form $h$.

\begin{proposition}\label{prop:factor}
We have
\begin{equation}\label{eq:factorization}
 H_\mu=N^*N.
\end{equation}
Its domain is
\begin{equation}\label{eq:operator-domain}
 \Dom(H_\mu)=\left\{a\in\Dom(N):
 \left(\int_{1/2}^\infty(N a)(s)n^{-s}\,d\mu(s)\right)_{n\ge1}\in\ell^2\right\}.
\end{equation}
For $a\in\Dom(H_\mu)$,
\begin{equation}\label{eq:operator-action}
 (H_\mu a)_n=\int_{1/2}^\infty(\ds a)(s)n^{-s}\,d\mu(s),\qquad n\ge1.
\end{equation}
\end{proposition}
\begin{proof}
Let $a\in\Dom(H_\mu)$. Then $a\in\Dom(h)=\Dom(N)$.
Since $e_n\in c_{00}\subset\Dom(h)$, the representation theorem
\eqref{eq:representation-action} and Theorem~\ref{thm:synthesis} give
\[
 \begin{aligned}
 (H_\mu a)_n
 &=\ip{H_\mu a}{e_n}_{\ell^2}=h[a,e_n]\\
 &=\ip{Na}{Ne_n}_{L^2(\mu)}
 =\int_{1/2}^\infty(N a)(s)n^{-s}\,d\mu(s).
 \end{aligned}
\]
The sequence of these integrals is $H_\mu a\in\ell^2$. Since $N^*=N_0^*$,
Lemma~\ref{lem:adjoint} implies that
\[
 Na\in\Dom(N^*),\qquad N^*Na=H_\mu a.
\]
Consequently,
\[
 H_\mu\subset N^*N.
\]
Both operators are self-adjoint, hence $H_\mu=N^*N$.

Therefore,
\[
 \Dom(H_\mu)=\{a\in\Dom(N):Na\in\Dom(N^*)\}.
\]
Lemma~\ref{lem:adjoint} now gives \eqref{eq:operator-domain}, while
\eqref{eq:operator-action} has already been established.
\end{proof}

\subsection{An operator core}
We shall prove that the subspace
\[
 \D_0=\left\{a\in c_{00}:\sum_{n\ge1}\frac{a_n}{\sqrt n}=0\right\}
\]
is an operator core for $H_\mu$. Its defining condition means that the Dirichlet
polynomial $\ds a$ vanishes at $s=1/2$. We have
\begin{equation*}
 \D_0=\operatorname{span}\{v_q:q\ge2\},\qquad v_q=e_1-\sqrt q\,e_q.
\end{equation*}
Indeed, $v_q\in\D_0$, and every $a\in\D_0$ can be written as
\[
 a=-\sum_{q\ge2}\frac{a_q}{\sqrt q}\,v_q.
\]
Moreover, $\D_0$ is dense in $\ell^2$: if $g\in\ell^2$ is orthogonal to every
$v_q$, then $g_q=g_1/\sqrt q$ for $q\ge2$. Since $\sum_{q\ge1}q^{-1}=\infty$,
this implies $g=0$.

For $t>0$, define the positive measure
\begin{equation*}
 d\mu_t(s)=q_t(s)\,d\mu(s),\qquad q_t(s)=1-e^{-t(s-1/2)},
\end{equation*}
and its moment sequence
\[
 \beta_t(n)=\int_{1/2}^\infty n^{-s}\,d\mu_t(s),\qquad n\ge1.
\]
We also set $q_0=0$. Let
\[
 \mathcal H_{\mu_t}=\bigl(\beta_t(mn)\bigr)_{m,n\ge1}
\]
be the corresponding Helson matrix. For $a,b\in c_{00}$,
\[
 \sum_{m,n\ge1}\beta_t(mn)a_n\overline{b_m}
 =\int_{1/2}^\infty(\ds a)(s)\overline{(\ds b)(s)}\,d\mu_t(s).
\]
By \eqref{eq:evaluation},
\[
 \begin{aligned}
 \left|\sum_{m,n\ge1}\beta_t(mn)a_n\overline{b_m}\right|
 &\le\norm a_{\ell^2}\norm b_{\ell^2}
       \int_{1/2}^\infty q_t(s)\zeta(2s)\,d\mu(s)\\
 &\le\norm a_{\ell^2}\norm b_{\ell^2}
       \int_{1/2}^\infty\left(q_t(s)+\frac{q_t(s)}{2s-1}\right)\,d\mu(s)\\
 &\le(1+t/2)\mu(\R)\norm a_{\ell^2}\norm b_{\ell^2}.
 \end{aligned}
\]
Thus $\mathcal H_{\mu_t}$ defines a bounded positive operator on $\ell^2$, with
\begin{equation}\label{eq:regular-bound}
 \norm{\mathcal H_{\mu_t}}
 \le\int_{1/2}^\infty\zeta(2s)\,d\mu_t(s)
 \le(1+t/2)\mu(\R).
\end{equation}
The same estimate shows that $a\mapsto\ds a$ is bounded from $\ell^2$ to
$L^2(\mu_t)$. Passing to the limit from finitely supported sequences therefore
gives, for $g\in\ell^2$ and $m\ge1$,
\begin{equation}\label{eq:regular-action}
 \begin{aligned}
 (\mathcal H_{\mu_t}g)_m
 &=\sum_{n\ge1}\beta_t(mn)g_n\\
 &=\int_{1/2}^\infty(\ds g)(s)m^{-s}\,d\mu_t(s).
 \end{aligned}
\end{equation}

\begin{lemma}\label{lem:core-action}
For every integer $q\ge2$,
\[
 v_q\in\Dom(H_\mu),\qquad H_\mu v_q=\mathcal H_{\mu_{\log q}}e_1.
\]
Consequently, $\D_0\subset\Dom(H_\mu)$.
\end{lemma}
\begin{proof}
Since $v_q\in c_{00}\subset\Dom(N)$ and
\[
 (\ds v_q)(s)=1-q^{1/2-s}=q_{\log q}(s),
\]
we have
\[
 \left(\int_{1/2}^\infty(\ds v_q)(s)n^{-s}\,d\mu(s)\right)_{n\ge1}
 =\mathcal H_{\mu_{\log q}}e_1\in\ell^2.
\]
Proposition~\ref{prop:factor} gives the assertion.
\end{proof}

\begin{proof}[Proof of Theorem~\ref{thm:core-main}]
The operator $T:=H_\mu|_{\D_0}$ is densely defined, symmetric and nonnegative.
It suffices to prove
\[
 \ker(T^*+I)=\{0\}.
\]

Let $g\in\ker(T^*+I)$, so that $T^*g=-g$. By Lemma~\ref{lem:core-action} and
self-adjointness of $\mathcal H_{\mu_{\log q}}$, for every integer $q\ge2$,
\begin{equation*}
 \begin{aligned}
 (\mathcal H_{\mu_{\log q}}g)_1
 &=\ip g{Tv_q}_{\ell^2}\\
 &=\ip{T^*g}{v_q}_{\ell^2}=\sqrt q\,g_q-g_1.
 \end{aligned}
\end{equation*}
Define
\begin{equation*}
 \begin{aligned}
 u(t)&=g_1+(\mathcal H_{\mu_t}g)_1\\
 &=g_1+\int_{1/2}^\infty q_t(s)(\ds g)(s)\,d\mu(s),\qquad t>0,
 \end{aligned}
\end{equation*}
and put $u(0)=g_1$.
Then
\begin{equation*}
 u(\log n)=\sqrt n\,g_n,\qquad n\ge1,
\end{equation*}
and hence
\begin{equation*}
 \sum_{n\ge1}\frac{|u(\log n)|^2}{n}=\norm g^2.
\end{equation*}
For $v>0$,
\[
 u'(v)=\int_{1/2}^\infty
 (s-1/2)e^{-v(s-1/2)}(\ds g)(s)\,d\mu(s),
\]
for $v\ge t>0$, the Cauchy--Schwarz inequality gives
\begin{equation*}
 \begin{aligned}
 |u'(v)|^2
 &\le\left(\int_{1/2}^\infty q_1(s)|(\ds g)(s)|^2\,d\mu(s)\right)\left(\int_{1/2}^\infty
       \frac{(s-1/2)^2e^{-2v(s-1/2)}}{q_1(s)}\,d\mu(s)\right)\\
 &\le\ip{\mathcal H_{\mu_1}g}g_{\ell^2}\,\mu(\R)
       \sup_{r>0}\frac{r^2e^{-2tr}}{1-e^{-r}}\\
 &\le\ip{\mathcal H_{\mu_1}g}g_{\ell^2}\,\mu(\R)
       \left(\frac1{2et}+\frac1{e^2t^2}\right).
 \end{aligned}
\end{equation*}
The last inequality follows from
\[
 \frac{r^2e^{-2tr}}{1-e^{-r}}\le(r+r^2)e^{-2tr},
 \qquad
 \sup_{r>0}re^{-2tr}=\frac1{2et},
 \qquad
 \sup_{r>0}r^2e^{-2tr}=\frac1{e^2t^2}.
\]
Consequently,
\[
 M_t:=\sup_{v\ge t}|u'(v)|<\infty,\qquad t>0,
\]
and $M_t\to0$ as $t\to\infty$.

We next prove that $u\in L^2(0,\infty)$. Put
\[
 \Delta_n=\log(n+1)-\log n.
\]
For $n\ge3$ and $v\in[\log n,\log(n+1)]$,
\[
 \begin{aligned}
 |u(v)|
 &\le|u(\log n)|+\int_{\log n}^v|u'(w)|\,dw\\
 &\le|u(\log n)|+M_1(v-\log n)\\
 &\le|u(\log n)|+M_1\Delta_n.
 \end{aligned}
\]
Thus
\[
 |u(v)|^2\le2|u(\log n)|^2+2M_1^2\Delta_n^2,
\]
and integration gives
\[
 \int_{\log n}^{\log(n+1)}|u(v)|^2\,dv
 \le2\Delta_n|u(\log n)|^2+2M_1^2\Delta_n^3.
\]
Summing over $n\ge3$ and using $\Delta_n\le1/n$, we obtain
\[
 \begin{aligned}
 \int_{\log3}^\infty|u(v)|^2\,dv
 &\le2\sum_{n\ge3}\frac{|u(\log n)|^2}{n}
       +2M_1^2\sum_{n\ge3}\frac1{n^3}\\
 &\le2\norm g^2+2M_1^2\sum_{n\ge3}\frac1{n^3}<\infty.
 \end{aligned}
\]
For $0<t\le\log3$, \eqref{eq:regular-bound} yields
\[
 \begin{aligned}
 |u(t)|&\le|g_1|+\norm{\mathcal H_{\mu_t}}\norm g\\
 &\le|g_1|+(1+\tfrac12\log3)\mu(\R)\norm g.
 \end{aligned}
\]
Hence $u$ is bounded on $[0,\log3]$, and therefore $u\in L^2(0,\infty)$.

\begin{samepage}
For $t>0$ and $n\ge1$, the identity
\[
 q_{t+\log n}(s)=q_{\log n}(s)+n^{-(s-1/2)}q_t(s)
\]
gives
\begin{equation*}
 \begin{aligned}
 u(t+\log n)
 &=u(\log n)+\int_{1/2}^\infty
       n^{-(s-1/2)}q_t(s)(\ds g)(s)\,d\mu(s)\\
 &=\sqrt n\bigl(g_n+(\mathcal H_{\mu_t}g)_n\bigr).
 \end{aligned}
\end{equation*}
\end{samepage}
It follows that
\[
 \norm{g+\mathcal H_{\mu_t}g}^2
 =\sum_{n\ge1}\frac{|u(t+\log n)|^2}{n}.
\]
Positivity of $\mathcal H_{\mu_t}$ implies
\[
 \begin{aligned}
 \norm{g+\mathcal H_{\mu_t}g}^2
 &=\norm g^2+2\ip{\mathcal H_{\mu_t}g}g_{\ell^2}
       +\norm{\mathcal H_{\mu_t}g}^2\\
 &\ge\norm g^2.
 \end{aligned}
\]
It remains to show that the preceding sum tends to zero.

For $t\ge1$ and $v\in[t+\log n,t+\log(n+1)]$,
\[
 \begin{aligned}
 |u(t+\log n)|
 &\le|u(v)|+\int_{t+\log n}^v|u'(w)|\,dw\\
 &\le|u(v)|+M_t(v-t-\log n)\\
 &\le|u(v)|+M_t\Delta_n.
 \end{aligned}
\]
Therefore
\[
 |u(t+\log n)|^2\le2|u(v)|^2+2M_t^2\Delta_n^2.
\]
Integrating over this interval, we obtain
\[
 \begin{aligned}
 \Delta_n|u(t+\log n)|^2
 &=\int_{t+\log n}^{t+\log(n+1)}|u(t+\log n)|^2\,dv\\
 &\le2\int_{t+\log n}^{t+\log(n+1)}|u(v)|^2\,dv
       +2M_t^2\Delta_n^3.
 \end{aligned}
\]
Since $1/n\le2\Delta_n$,
\[
 \frac{|u(t+\log n)|^2}{n}
 \le4\int_{t+\log n}^{t+\log(n+1)}|u(v)|^2\,dv+4M_t^2\Delta_n^3.
\]
Summing over $n\ge1$ gives
\begin{equation*}
 \begin{aligned}
 \norm{g+\mathcal H_{\mu_t}g}^2
 &\le4\int_t^\infty|u(v)|^2\,dv+4M_t^2\sum_{n\ge1}\Delta_n^3\\
 &\longrightarrow0\qquad(t\to\infty).
 \end{aligned}
\end{equation*}
Thus $\norm g^2=0$, proving essential self-adjointness of $T$. Since $H_\mu$ is
a self-adjoint extension of $T$, we conclude that $\overline T=H_\mu$.
Hence $\D_0$ is an operator core for $H_\mu$.
\end{proof}

\subsection{Bounded approximation}
The following theorem characterizes $\Dom(H_\mu)$ in terms of the bounded Helson
matrices  $\mathcal H_{\mu_t}$ and shows that $\mathcal H_{\mu_t}g\to H_\mu g$
for every $g\in\Dom(H_\mu)$.

\begin{theorem}\label{thm:regular-domain}
For every $g\in\ell^2$,
\begin{equation}\label{eq:domain-main}
 \begin{split}
 g\in\Dom(H_\mu)&\iff\sup_{t>0}\norm{\mathcal H_{\mu_t}g}<\infty\\
 &\iff\mathcal H_{\mu_t}g\text{ converges in }\ell^2\text{ as }t\to\infty.
 \end{split}
\end{equation}
For $g\in\Dom(H_\mu)$, the limit in \eqref{eq:domain-main} equals $H_\mu g$.
\end{theorem}
\begin{proof}
Let $g\in\Dom(H_\mu)$ and put $f=Ng$. By \eqref{eq:factorization} and
\eqref{eq:graph-equivalence}, $f\in\Dom(N^*)=\Dom(L)$. For $t>0$, set
\[
 f_t(s)=e^{-t(s-1/2)}f(s).
\]
The dominated convergence theorem gives $f_t\to0$ in $L^2(\mu)$. Moreover,
\[
 (Lf_t)(r)=(Lf)(r+t),\qquad r\ge0,
\]
so that $f_t\in\Dom(L)=\Dom(N^*)$ and
\[
 \norm{Lf_t}_{L^2(0,\infty)}^2=\int_t^\infty|Lf(v)|^2\,dv\longrightarrow0.
\]
It follows from \eqref{eq:graph-equivalence} that
\[
 \norm{N^*f_t}\longrightarrow0.
\]
The coordinate formulas \eqref{eq:operator-action} and \eqref{eq:regular-action}
imply
\[
 H_\mu g-\mathcal H_{\mu_t}g=N^*f_t.
\]
Thus $\mathcal H_{\mu_t}g\to H_\mu g$ in $\ell^2$ as $t\to\infty$.

Conversely, suppose
\[
 \sup_{t>0}\norm{\mathcal H_{\mu_t}g}\le C.
\]
Put $T=H_\mu|_{\D_0}$. For $v\in\D_0$, the convergence established above and
self-adjointness of $\mathcal H_{\mu_t}$ give
\[
 |\ip g{Tv}|=\lim_{t\to\infty}|\ip{\mathcal H_{\mu_t}g}v|\le C\norm v.
\]
Hence $g\in\Dom(T^*)$. Theorem~\ref{thm:core-main} gives $T^*=H_\mu$, and therefore
$g\in\Dom(H_\mu)$.

Finally, if $\mathcal H_{\mu_t}g$ converges in $\ell^2$ as $t\to\infty$, its norm is
bounded for sufficiently large $t$. The estimate \eqref{eq:regular-bound} gives
boundedness on every interval $0<t\le R$, so that
$\sup_{t>0}\norm{\mathcal H_{\mu_t}g}<\infty$.
\end{proof}

\subsection{Trace class and the maximal matrix realization}\label{subsec:matrix}
We first record a trace-class criterion. The diagonal entries $\alpha(n^2)$ play
the role of the even moments in \cite[Theorem~3.1]{BS22}. A related positive
diagonal criterion appears in \cite[Appendix, Proposition~A(i)]{CSS26}.

\begin{samepage}
\begin{theorem}\label{thm:trace}
Let $\mu$ be a finite positive Borel measure on $\R$ satisfying
\eqref{eq:moments}. The following conditions are equivalent:
\begin{enumerate}
\item\label{item:trace-weighted} $\displaystyle\sum_{n\ge1}\frac{\alpha(n)}{\sqrt n}<\infty$;
\item\label{item:trace-diagonal} $\displaystyle\sum_{n\ge1}\alpha(n^2)<\infty$;
\item\label{item:trace-integral} $\mu((-\infty,1/2])=0$ and
      $\displaystyle\int_{1/2}^\infty\zeta(2s)\,d\mu(s)<\infty$;
\item\label{item:trace-operator} The form $h_0$ is represented by a positive trace-class operator
$H_\mu$ on $\ell^2$.
\end{enumerate}
When these conditions hold,
\begin{equation*}
 \Tr H_\mu=\sum_{n\ge1}\alpha(n^2)=\int_{1/2}^\infty\zeta(2s)\,d\mu(s),
 \qquad\norm{H_\mu}\le\sum_{n\ge1}\frac{\alpha(n)}{\sqrt n}.
\end{equation*}
\end{theorem}
\end{samepage}
\begin{proof}
Tonelli's theorem gives
\[
 \begin{aligned}
 \sum_{n\ge1}\alpha(n^2)
 &=\int_{\R}\sum_{n\ge1}n^{-2s}\,d\mu(s),\\
 \sum_{n\ge1}\frac{\alpha(n)}{\sqrt n}
 &=\int_{\R}\sum_{n\ge1}n^{-(s+1/2)}\,d\mu(s).
 \end{aligned}
\]
Both inner series diverge for $s\le1/2$, so finiteness of either sum implies
\eqref{eq:concentration}. The first identity therefore proves the equivalence
of (\ref{item:trace-diagonal}) and (\ref{item:trace-integral}).

For $s>1/2$, comparison with the integrals of the corresponding decreasing
functions gives
\[
 \begin{aligned}
 \zeta(s+1/2)
 &\le1+\int_1^\infty x^{-s-1/2}\,dx
 =1+\frac{2}{2s-1},\\
 \zeta(2s)
 &\ge\int_1^\infty x^{-2s}\,dx
 =\frac{1}{2s-1}.
 \end{aligned}
\]
Using also $\zeta(2s)\ge1$, we obtain
\[
 \zeta(2s)\le\zeta(s+1/2)\le3\zeta(2s).
\]
Together with the second Tonelli identity, this proves the equivalence of
(\ref{item:trace-weighted}) and (\ref{item:trace-integral}).

If (\ref{item:trace-diagonal}) holds, then
\[
 \sum_{n\ge1}\norm{N_0e_n}_{L^2(\mu)}^2
 =\sum_{n\ge1}\alpha(n^2)<\infty.
\]
Hence $N_0$ extends to a Hilbert--Schmidt operator $N=\overline{N_0}$ and $H_\mu=N^*N$ is of trace
class, proving (\ref{item:trace-operator}). Conversely, (\ref{item:trace-operator}) gives
\[
 \sum_{n\ge1}\alpha(n^2)
 =\sum_{n\ge1}\ip{H_\mu e_n}{e_n}
 =\Tr H_\mu<\infty.
\]
This proves the remaining equivalence and the trace formula.

Finally, positivity of $H_\mu$ and the trace formula give
\[
 \norm{H_\mu}
 \le \Tr H_\mu
 =\int_{1/2}^\infty\zeta(2s)\,d\mu(s)
 \le\int_{1/2}^\infty\zeta(s+1/2)\,d\mu(s)
 =\sum_{n\ge1}\frac{\alpha(n)}{\sqrt n}.
\]
\end{proof}

Assume that $(\alpha(n))_{n\ge1}\in\ell^2$. The Helson matrix
\[
 \mathcal H_\mu=(\alpha(mn))_{m,n\ge1}
\]
has square-summable rows and columns, since
\[
 0\le\alpha(mn)\le\alpha(n),\qquad m,n\ge1.
\]
For every $a\in\ell^2$, the matrix product is therefore defined coordinatewise by
\begin{equation*}
 (\mathcal H_\mu a)_m=\sum_{n\ge1}\alpha(mn)a_n,\qquad m\ge1,
\end{equation*}
with absolute convergence of each series. We consider the operator
$(\mathcal H_\mu,D_{\max})$, where
\begin{equation*}
 D_{\max}=\{a\in\ell^2:\mathcal H_\mu a\in\ell^2\}.
\end{equation*}
The following result is the analogue of \cite[Theorem~3.8]{BS22} for Helson forms.

\begin{theorem}\label{thm:matrix}
If $(\alpha(n))_{n\ge1}\in\ell^2$, then the operator
$(\mathcal H_\mu,D_{\max})$ coincides with $H_\mu$.
\end{theorem}
\begin{proof}
Set $S=\mathcal H_\mu|_{c_{00}}$. Since each column of $\mathcal H_\mu$ belongs
to $\ell^2$, Proposition~\ref{prop:factor} gives
\[
 \qquad H_\mu e_n=\mathcal H_\mu e_n,\qquad n\ge1.
\]
Thus $S=H_\mu|_{c_{00}}$. By Theorem~\ref{thm:core-main}, $\D_0$ is an
operator core for $H_\mu$. Since
\[
 \D_0\subset c_{00}\subset\Dom(H_\mu),
\]
we have $\overline S=H_\mu$.

For $a\in\Dom(S^*)$,
\[
 (S^*a)_m=\ip a{Se_m}
 =\sum_{n\ge1}\alpha(mn)a_n
 =(\mathcal H_\mu a)_m,\qquad m\ge1.
\]
Hence $a\in D_{\max}$. Conversely, if $a\in D_{\max}$, then
\[
 \ip{Sb}a=\ip b{\mathcal H_\mu a},\qquad b\in c_{00},
\]
so $a\in\Dom(S^*)$ and $S^*a=\mathcal H_\mu a$. Therefore
\[
 (\mathcal H_\mu,D_{\max})=S^*=(\overline S)^*=H_\mu.
\]
\end{proof}

By Tonelli's theorem, the condition $(\alpha(n))_{n\ge1}\in\ell^2$ is equivalent to
\begin{equation*}
 \sum_{n\ge1}\alpha(n)^2
 =\iint_{(1/2,\infty)^2}\zeta(s+u)\,d\mu(s)\,d\mu(u)<\infty.
\end{equation*}
If $(\alpha(n))_{n\ge1}\notin\ell^2$, Proposition~\ref{prop:factor} gives
$e_1\notin\Dom(H_\mu)$.

\begin{example}\label{ex:gamma}
For $c>0$, let $\mu_c$ be the probability measure on $(1/2,\infty)$ given by
\[
 d\mu_c(s)=\frac{(s-1/2)^{c-1}e^{-(s-1/2)}}{\Gamma(c)}\,ds.
\]
Its moments are
\[
 \alpha_c(n)=\int_{1/2}^{\infty}n^{-s}\,d\mu_c(s)
 =\frac1{\sqrt n(1+\log n)^c},\qquad n\ge1.
\]
Let $h_c$ denote the closed Helson form corresponding to $\mu_c$, and let
$H_c=H_{\mu_c}$ be its associated positive self-adjoint operator. By
\cite[Theorem~5.1]{PP18a}, $H_c$ is bounded for $c\ge1$ and unbounded for
$0<c<1$. Theorem~\ref{thm:trace} shows that $H_c$ is of trace class if and only
if $c>1$. Moreover, $(\alpha_c(n))_{n\ge1}\in\ell^2$ if and only if $c>1/2$;
in this range, Theorem~\ref{thm:matrix} identifies $H_c$ with
$(\mathcal H_{\mu_c},D_{\max})$. At $c=1$, the matrix $\mathcal H_{\mu_1}$ is
the regularized multiplicative Hilbert matrix studied in
\cite{BPSSV16,PP18b}.

For $\mu=\mu_c$ with $0<c<1$, Theorem~\ref{thm:synthesis} and
Proposition~\ref{prop:factor} give
\[
 \Dom(h_c)=\left\{a\in\ell^2:
 \sum_{n\ge1}a_nn^{-s}\in L^2(\mu_c)\right\}
\]
and
\[
 \Dom(H_c)=\left\{a\in\Dom(h_c):
 \left(\int_{1/2}^{\infty}(\ds a)(s)m^{-s}\,d\mu_c(s)\right)_{m\ge1}
 \in\ell^2\right\}.
\]
For $d>1/2$, define $g_d\in\ell^2$ by
\[
 g_d(n)=\frac1{\sqrt n(1+\log n)^d},\qquad n\ge1.
\]
Then
\[
 g_d\in\Dom(h_c)\iff d>1-\frac c2,
\]
whereas
\[
 g_d\in\Dom(H_c)\iff c>\frac12\quad\text{and}\quad d>\frac32-c.
\]
Consequently, if $1/2<c<1$, then
\[
 g_d\in\Dom(h_c)\setminus\Dom(H_c)
 \iff 1-\frac c2<d\le\frac32-c.
\]
If $0<c\le1/2$, the same domain difference contains $g_d$ precisely when
$d>1-c/2$. Thus
\[
 \Dom(H_c)\subsetneq\Dom(h_c),\qquad 0<c<1,
\]
while both domains equal $\ell^2$ for $c\ge1$.
\end{example}
\appendix
\section{}
\label{app:hankel-maximal}

We give another proof of the maximal-domain description for
Hankel moment operators; see~\cite{BS22, Y16}.
The argument is analogous to the proof of Theorem~\ref{thm:synthesis}.

Let $\mu$ be a finite positive Borel measure concentrated on $(-1,1)$.
Consider the operator
\[
A_0:\ell^2(\mathbb N_0)\supset c_{00}(\mathbb N_0)
\longrightarrow L^2(\mu),\qquad
(A_0a)(x)=\sum_{n=0}^{\infty}a_nx^n,
\]
and its maximal realization
\[
\begin{aligned}
\operatorname{Dom}(A_{\max})
&=\left\{a\in\ell^2(\mathbb N_0):
\sum_{n=0}^{\infty}a_nx^n\in L^2(\mu)\right\},\\
(A_{\max}a)(x)&=\sum_{n=0}^{\infty}a_nx^n.
\end{aligned}
\]
For $a\in\ell^2(\mathbb N_0)$,
\[
\left|\sum_{n=0}^{\infty}a_nx^n\right|
\leq \frac{\|a\|_{\ell^2}}{\sqrt{1-x^2}},
\qquad -1<x<1.
\]
Thus convergence in $\ell^2$ implies pointwise convergence of the
corresponding power series on $(-1,1)$. If $a^{(j)}\to a$ in $\ell^2$
and $A_{\max}a^{(j)}\to g$ in $L^2(\mu)$, passage to a subsequence
converging $\mu$-almost everywhere gives
$g(x)=\sum_{n=0}^{\infty}a_nx^n$ almost everywhere.
Hence $A_{\max}$ is closed. Since $A_0\subset A_{\max}$, we have
$\overline{A_0}\subset A_{\max}$.

To prove the reverse inclusion, it suffices to verify that
\begin{equation}\label{eq:hankel-interchange}
\begin{aligned}
&\int_{-1}^{1}
\left(\sum_{n=0}^{\infty}a_nx^n\right)
\overline{f(x)}\,d\mu(x)\\
&\qquad=
\sum_{n=0}^{\infty}a_n
\int_{-1}^{1}x^n\overline{f(x)}\,d\mu(x),
\end{aligned}
\end{equation}
for $a\in\operatorname{Dom}(A_{\max})$ and
$f\in\operatorname{Dom}(A_0^*)$.
Fix such $a$ and $f$, let
\[
b_n=(A_0^*f)_n=\int_{-1}^{1}x^nf(x)\,d\mu(x),
\qquad n\geq0.
\]
so that $(b_n)_{n\geq0}\in\ell^2$, and let
\[
q_N(x)=1-x^{2N},\qquad N\geq1.
\]
Since $0\leq q_N(x)\leq1$ for $-1<x<1$ and
\[
\frac{q_N(x)^2}{1-x^2}
=(1-x^{2N})\sum_{j=0}^{N-1}x^{2j}\leq N,
\]
the Cauchy--Schwarz inequality and Tonelli's theorem give
\begin{align*}
&\sum_{n=0}^{\infty}|a_n|
\int_{-1}^{1}|x|^nq_N(x)|f(x)|\,d\mu(x)\\
&\quad\leq
\|a\|_{\ell^2}
\int_{-1}^{1}
\frac{q_N(x)|f(x)|}{\sqrt{1-x^2}}\,d\mu(x)\\
&\quad\leq
\sqrt{N\mu((-1,1))}\,
\|a\|_{\ell^2}\|f\|_{L^2(\mu)}<\infty.
\end{align*}
Consequently, the sum and the integral can be interchanged after
multiplication by $q_N$, and
\begin{equation}\label{eq:hankel-cutoff}
\begin{aligned}
&\int_{-1}^{1}
\left(\sum_{n=0}^{\infty}a_nx^n\right)
q_N(x)\overline{f(x)}\,d\mu(x)\\
&\qquad=
\sum_{n=0}^{\infty}a_n\overline{(b_n-b_{n+2N})}.
\end{aligned}
\end{equation}
Since $q_N(x)\to1$ for $-1<x<1$, dominated convergence theorem gives
$q_Nf\to f$ in $L^2(\mu)$. Moreover,
\[
\left|\sum_{n=0}^{\infty}a_n\overline{b_{n+2N}}\right|
\leq\|a\|_{\ell^2}
\left(\sum_{k=2N}^{\infty}|b_k|^2\right)^{1/2}
\longrightarrow0.
\]
Letting $N\to\infty$ in~\eqref{eq:hankel-cutoff}, we obtain
\eqref{eq:hankel-interchange}. Therefore
$A_{\max}\subset A_0^{**}=\overline{A_0}$, and hence
\[
\overline{A_0}=A_{\max}.
\]

\vspace{2em}
\subsection*{AI disclosure}
The authors used GPT-6 (OpenAI) to assist with this work. The  construction in Example~\ref{ex:counter} was provided by GPT-6. The methods used to prove the remaining results
were developed by the authors. The authors independently reviewed all
results and proofs and take full responsibility for the mathematical
content of the paper.

\vspace{2em}
\subsection*{Acknowledgment}
Y. Lu was supported by the National Natural Science Foundation of China
(Grant No. 12031002 and No. 12671149). C. Zu was supported by the National Natural Science Foundation of China (Grant No. 12401151).

\subsection*{Conflict of interest}
The authors declare no relevant conflicts of interest regarding this article. 
%The authors declare no conflicts of interest.
\subsection*{Data availability statement}
No data, models, or code were generated or used during the study. 
%Data sharing is not applicable to this article as no new data were created or analyzed in this study.

\bibliographystyle{amsplain}
\bibliography{ref}
\end{document}